\documentclass[12pt,reqno,twoside]{amsart}

\usepackage{amsfonts,amsmath,amssymb}
\usepackage{mathrsfs,mathtools}
\usepackage{enumerate}
\usepackage{hyperref}
\usepackage{esint}
\usepackage{graphicx}
\usepackage{bm}
\usepackage{commath}
\usepackage{esint}
\DeclareMathAlphabet{\mathpzc}{OT1}{pzc}{m}{it}

\usepackage{placeins} 
\usepackage{flafter} 

\usepackage[textsize=small]{todonotes}
\usepackage{caption}

\numberwithin{equation}{section}

\usepackage[dvips,bottom=1.4in,right=1in,top=1in, left=1in]{geometry}

\makeatletter
\def\eqnarray{\stepcounter{equation}\let\@currentlabel=\theequation
\global\@eqnswtrue
\tabskip\@centering\let\\=\@eqncr
$$\halign to \displaywidth\bgroup\hfil\global\@eqcnt\z@
  $\displaystyle\tabskip\z@{##}$&\global\@eqcnt\@ne
  \hfil$\displaystyle{{}##{}}$\hfil
  &\global\@eqcnt\tw@ $\displaystyle{##}$\hfil
  \tabskip\@centering&\llap{##}\tabskip\z@\cr}

\def\endeqnarray{\@@eqncr\egroup
      \global\advance\c@equation\m@ne$$\global\@ignoretrue}

\newtheorem{theorem}{Theorem}[section]

\newtheorem{lemma}[theorem]{Lemma}

\newtheorem{problem}[theorem]{Problem}

\newtheorem{remark}[theorem]{Remark}
\numberwithin{equation}{section}

\newcommand{\Om}{\Omega}

\newcommand{\V}{\mathcal{V}}

\newcommand{\Bl}[1]{a\left( #1\right)}
\newcommand{\Ian}[1]{\left\langle #1\right\rangle}

\newcommand{\on}{~\text{ on }~}

\title[Elliptic Reconstruction for Parabolic Variational Inequalities]{Elliptic Reconstruction and A Posteriori Error Estimates for Parabolic Variational Inequalities}

	\author{Harbir Antil}
\address{Harbir Antil, Department of Mathematical Sciences and the Center for Mathematics and Artificial Intelligence (CMAI), George Mason University,  Fairfax, VA 22030, USA.}
\email{hantil@gmu.edu}

\author{Rohit Khandelwal}
\address{Rohit Khandelwal, Department of Mathematical Sciences and the Center for Mathematics and Artificial Intelligence (CMAI), George Mason University,  Fairfax, VA 22030, USA.}
\email{rkhandel@gmu.edu}

\thanks{This work is partially supported by NSF grant DMS-2110263, Air Force Office of Scientific
			Research (AFOSR) under Award NO: FA9550-22-1-0248, and Office of Naval Research (ONR) 
			under Award NO: N00014-24-1-2147.}



\begin{document}

\begin{abstract}
Elliptic reconstruction property, originally introduced by Makridakis and Nochetto for linear parabolic problems, is a well-known tool to derive optimal a posteriori error estimates. No such results are known for nonlinear and nonsmooth problems such as parabolic variational inequalities (VIs). This article establishes the elliptic reconstruction property for parabolic VIs and derives a posteriori error estimates in $L^{\infty}(0,T;L^{2}(\Omega))$. The estimator consists of discrete complementarity terms and standard residual. As an application, the residual-type error estimates are presented.
\end{abstract}

\maketitle

\section{Introduction} \label{sec1}

Parabolic variational inequalities (VIs) arise in many applications. Examples include,
Stefan problem for phase transition \cite{MR0737005} and American options in 
finance \cite{wilmott1993option}. Designing numerical schemes to approximate such  
problems is challenging due to the problems being nonlinear and nonsmooth nature 
of the continuous solution. In 
addition, these numerical schemes must be constraint preserving, for instance bound constraints,
or the constraint violation error must remain under control. The focus of this paper is on 
such challenges. Notice that solving these problems in itself is a challenge, which 
we do not address here, see for instance \cite{MR0737005,MHintermueller_KIto_KKunisch_2002a}. 

During last two decades, significant research focus has been on adaptive finite element methods (AFEM) for elliptic partial differential equations (PDEs) \cite{ainsworth2011posteriori,verfurth1996review}. Starting from \cite{eriksson1991adaptive,eriksson1995adaptive}, there have been advances in AFEM for parabolic problems as well, but not to the extent of elliptic problems.  See also \cite{MR2772095} for instance for AFEM for Allen-Cahn and \cite{MR4593215} for AFEM using convex duality. AFEM popularity can been attributed to its ability to capture singularities and minimize computational costs. A key mathematical tool in the development of efficient adaptive versions of numerous numerical schemes is the use of a posteriori estimates. These estimates provide a fresh perspective on the theoretical exploration of a scheme's behavior. This becomes particularly valuable in situations where ``reasonable discretizations" may not consistently yield expected results. The situation is more delicate when it comes to VIs. Several contributions exists on a posteriori estimation of elliptic VIs: Ainsworth, Oden, and Lee \cite{ainsworth1993local}, Hoppe and Kornhuber \cite{hoppe1994adaptive}, Chen and Nochetto \cite{chen2000residual}, Veeser \cite{veeser2001efficient}, Bartels and Carstensen \cite{bartels2004averaging} and Nochetto, Petersdorff and Zhang \cite{nochetto2010posteriori}. The works related to a posteriori error estimation of the parabolic VIs are significantly limited. To the best of authors' knowledge, the articles \cite{moon2007posteriori,achdou2008posteriori} are only available works which are related to the adaptive FEM of parabolic VIs in the energy norm.

The goal of this work is to introduce a continuous {\it elliptic reconstruction} for parabolic VIs and use it to derive abstract a posteriori error 
estimates in $L^{\infty}(0,T;L^{2}(\Omega))$ norm 
for the semidiscrete finite element approximations, of any degree $k \geq 1$, to parabolic VIs. This construction is motivated by Makridakis and Nochetto \cite{makridakis2003elliptic} where the authors developed this reconstruction property for linear parabolic problems. See also \cite{MR2772095} for instance for AFEM for Allen-Cahn and \cite{MR4593215} for AFEM using convex duality. In the setting of the current paper, the introduced reconstruction, $\mathcal{W}(t): \V_h \rightarrow \V$ of the discrete solution $w_h(t)$, where $t \in [0,T]$, also accounts for the Lagrange multipliers corresponding to the parabolic VIs. Notice that here $\V$ and $\V_h$ are the continuous and discrete spaces, respectively. 

The key observation in our analysis is the following inequality (Lemma \ref{pointerror}) for the difference between elliptic reconstruction $\mathcal{W}$ and continuous solution $w$:
\begin{align} \label{pointwise}
(\mathcal{W}-w)_t -\Delta (\mathcal{W}-w) \leq \sigma_h + (\mathcal{W}-w_h)_t
\end{align}
where $\sigma_h \in \V_h$ is the discrete Lagrange multiplier \cite{fierro2003posteriori,nochetto2005fully}. Notice that, similarly to the continuous setting, 
$\sigma_{h}$ converts a discrete parabolic VI into an equality. For every $t$ when you think of the discrete VI as an elliptic (equality) problem, then the elliptic reconstruction $\mathcal{W}$ 
is precisely the continuous representation of solution to this discrete problem for every $t$. The right-hand-side of the continuous elliptic reconstruction problem depends on computable quantities $w_h$, the data $f$, $\Omega$ and $\sigma_{h}$. 
Thus, any known residual type a posteriori error estimates which are valid for elliptic problems on $\Omega$ can now be used to estimate $\mathcal{W}-w_h$.

\medskip
\noindent
{\bf Literature Survey:} 
Let us now discuss the origins of elliptic reconstruction. Recall the linear parabolic 
problem 
\begin{problem} \label{p1}Find $u: [0,T] \rightarrow \Omega$ such that
	\begin{align*}
	&u_t -\Delta u =f \quad  \text{in } \Omega \times (0,T], \\ &u(\cdot, 0)= u_0(\cdot)  \quad \text{in } \Omega, \\ &u=0 \quad \text{on }\partial \Omega \times (0,T] .
	\end{align*} 
	\end{problem}
Wheeler initially carried out the a priori error analysis for Problem \ref{p1} in the energy norm in \cite{wheeler1973priori}. The main idea behind this analysis is the elliptic projection operator $U$ which for any $t$ satisfies 
\begin{align} \label{k1}
\int_{\Omega} \nabla (U-u) \cdot \nabla v_h dx =0 \quad \forall v_h \in V_h \, .
\end{align}
It was further established that $U$ is the Galerkin approximation in $V_h$ to the weak solution of the corresponding linear elliptic problem
	\begin{align*}
-\Delta u &=f-u_t \quad  \text{in } \Omega, \\ u&=0 \qquad \text{on }\partial \Omega.
\end{align*}
Thus any a priori error estimates which are valid for elliptic problems on $\Omega$ can now be used to estimate $U-u_h$. The overall error $u-u_h$ for Problem \ref{p1} can then be derived by first estimating $u-U$ using standard PDE techniques and then estimating $U-u_h$ using elliptic a priori error analysis. The subsequent $L^2$-error estimates derived in \cite{wheeler1973priori} are of optimal order.

Later on, Makridakis and Nochetto \cite{makridakis2003elliptic} introduced the elliptic reconstruction map which can be regarded as an ``a posteriori" analogue to the Ritz-Wheeler projection defined in \eqref{k1}. The authors in \cite{makridakis2003elliptic} consider semidiscrete finite element approximations of any degree $k \geq 1$ and proved the optimal order a posteriori error estimates in $L^{\infty}(0,T; H^1_0(\Omega))$. They derived the pointwise equation which is satisfied by $u-U$, that is further used to compute estimates in terms of $u_{h,t}-U_t$. The key observation in their analysis is that $u_h$ is the discrete solution of the elliptic problem whose continuous solution is the elliptic reconstruction. In the context of elliptic reconstruction, we refer to \cite{lakkis2006elliptic} for the a posteriori error analysis of the fully discrete finite element approximations to Problem \ref{p1} in $L^{\infty}(0,T; H^1_0(\Omega))$.

\medskip
\noindent
{\bf Outline:} The article is organized as follows: section \ref{sec2} focuses on the problem formulation, i.e., parabolic VI. In section \ref{sec3}, we introduce the discretization and provide the semidiscrete version of parabolic VI. The elliptic reconstruction operator $\mathcal{W}$ is defined in section \ref{sec4} and therein, we prove the key orthogonality relation between $\mathcal{W}$ and discrete solution $w_h$. In section \ref{sec5}, we begin with the proof of inequality \eqref{pointwise} and we derive the abstract a posteriori error estimates in the energy norm. 
Lastly, we employ the standard residual-type error estimates and compute the explicit form of the error estimator in section \ref{sec6}.

\section{Problem Formulation} \label{sec2}

Let $\mathcal{V}:= H^1_0(\Om)$ and $Q:=\Om\times (0,T]$ where $\Omega \subset \mathbb{R}^d~;~d \in \{2,3\}$ is an open and bounded domain with boundary $\partial \Omega$ and $T>0$ is any fixed time. Let $f \in L^2(0, T; L^2(\Om))$ and $w_0 \in L^2(\Om)$ be the given data and $\chi \in H^1(Q)$ be a given obstacle such that $\chi \leq 0$ on $[0, T] \times \partial \Om$. We denote 
\begin{align*}
\mathcal{K}_{\chi}(t):=\left\{v\in \mathcal{V} \ |\ v\geq \chi(t)  \text{ a.e. in } \Omega \right\} \text{ a.e. } t \in [0,T]
\end{align*} to be a non-empty, closed and convex subset of $\V$. Then the continuous parabolic VI reads as follows: we seek a map $w:[0,T] \longrightarrow \mathcal{V}$ such that $w \in L^2(0,T;H^1(\Omega)) \cap C^0([0,T]; L^2(\Omega))$ and $w_t \in L^2(0,T; \mathcal{V}^*)$ ($\mathcal{V}^*$ is the dual of $\mathcal{V}$) (cf, \cite{MR0428137,MR0737005}) where $\mathcal{K} :=\left\{v\in L^2(0, T; \mathcal{V}) \ |\ v(t) \in  \mathcal{K}_{\chi}(t)~ \text{for a.e.~} t \in [0,T] \right\}$ and the following holds a.e. on $[0, T]$ 
	\begin{align}
		&\Ian{w_t(t),v-w(t)}_{-1,1}+\Bl{w(t),v-w(t)}\geq (f(t),v-w(t))~~ \forall \ v \in \mathcal{K}_{\chi}(t), \label{CVI}\\
		&w(0)=w_0(x) \on \Om \quad \text{and }  \quad w(t)\in \mathcal{K}_{\chi}(t) \text{ a.e.~} t \in [0, T],
		\label{CVI1}
	\end{align}
	where $a(p,q):= \int_{\Om} \nabla p \cdot \nabla q~dx~ \forall~ p, q \in \mathcal{V}$,  $\Ian{\cdot,\cdot}_{-1,1}$ denotes the duality pairing between $\mathcal{V}$ and $\mathcal{V}^*$ and $(\cdot,\cdot)$ is the $L^2(\Omega)$ inner-product. For our analysis, it is important to note that there exists a continuous Lagrange multiplier $\sigma \in L^2(0,T; \mathcal{V}^*)$ such that $\sigma(t)$ satisfies the following equation:
	\begin{align} \label{def:sigma}
	\Ian{\sigma(t),v}_{-1,1}=\Ian{w_t(t),v}_{-1,1}+\Bl{w(t),v}-(f(t),v)~~ \forall \ v\in \mathcal{V} \,, \text{ a.e.~} t \in (0, T].
	\end{align}
	\begin{remark}\label{rem:sign}
\rm 	
		Let $\phi \in \mathcal{V}$ be such that $\phi \geq 0$, and let $v=w(t)+\phi \in \mathcal{K}_{\chi}(t)$ in \eqref{CVI}, then from \eqref{def:sigma} it holds that $\sigma \geq 0,$ i.e., $\Ian{\sigma(t),\phi}_{-1,1} \geq 0~~\forall~ 0 \leq \phi \in \mathcal{V}$ and for a.e. $t \in (0,T]$. In the sense of distributions, we observe that $\sigma= w_t -f -\Delta w \geq 0$ in $\mathcal{C}:=\{w=\chi\}$ and $\sigma=0$ in $\mathcal{N}:=\{w > \chi\}$.\\
		Using \cite[pp. 96 and 100]{MR0428137} and \cite{MR4029105}, and if  
		$\Omega$ is bounded and open subset with sufficiently smooth boundary, $f \in L^2(0,T;L^2(\Omega))$, $w_0 \in H^1_0(\Omega)$, $\chi$ sufficiently smooth, then 
		$w$
		solving \eqref{CVI} additionally fulfills $w \in   C^0([0,T]; \V) \cap L^2(0,T; H^2(\Omega))$ and $w_t \in L^2(0,T; L^2(\Omega))$. Because of this, from \eqref{def:sigma}, we have that $\sigma \in L^2(0,T;L^2(\Omega))$. Thus $\sigma \ge 0$, a.e. in $Q$, we will use this to prove our key Lemma~\ref{pointerror}. 
	\end{remark}
	\begin{remark}\label{rem:comp}
	\rm 
	From the theory of variational inequalities \cite{glowinski1980numerical,achdou2008posteriori}, we next 
	state the complementarity conditions for a.e. $t \in (0,T]$:
	\begin{align*}
	\sigma \geq 0,~w(t) \geq \chi(t) \mbox{ an d }  \Ian{\sigma(t), w(t)-\chi(t)}_{-1,1}=0.
	\end{align*}
\end{remark}
\section{Notations and Discretization} \label{sec3}
Let $\Om_{h}$ be shape regular $d$-simplicial triangulations of $\Om \subset \mathbb{R}^d$. Denote by $K \in \Om_{h}$ a $d$-simplex that is a triangle in $\mathbb{R}^2$ or a tetrahedron in $\mathbb{R}^3$ and $|K|$ indicates the $d$-dimensional Lebesgue measure of $K$. The discrete space $\V_{h}$ is defined by
\begin{align}
\V_{h}:= \{v_h \in \mathcal{V} \, : \, v_h|_K \in \mathbb{P}_k(K)~\forall~ K \in \Om_{h}\} ,
\end{align}
where $ \mathbb{P}_k(K)$ denote polynomials of degree less than or equal to $k$ and $k \geq 1$ be an integer. Let us introduce the non-empty, closed and convex subset of $\V_{h}$: 
\begin{align}
\mathcal{K}_{h,\chi}(t):= \{v_h \in \V_{h},~v_h|_K \geq \chi_h|_K(t)~ \forall K \in \Omega_h \} \text{ a.e. } t \in [0,T],
\end{align}
where $\chi_h$ is the approximation (for instance, Lagrange \cite{brenner2007mathematical}) of $\chi$ and satisfying $\chi_h(t) \leq 0$ on $[0,T] \times \partial \Om$. The set of all vertices of triangle $K$ is denoted by $\V_K$. Let $\omega_e$ be the union of all $d$-simplices sharing the edge $e$. The set of all egdes of $K$ is denoted by $\mathcal{E}_{h,K}$. Denote the set of all vertices of a triangle of $\Om_{h}$ that are in $\Om$ (resp. on $\partial \Om$) by $\mathcal{N}^i_{h}$ (resp. $\mathcal{N}^b_{h}$). Set $\mathcal{N}_{h}:=\mathcal{N}^i_{h} \cup \mathcal{N}^b_{h}$. Define $\mathcal{E}_{h}:=\mathcal{E}^i_{h} \cup \mathcal{E}^b_{h}$ is the set of all edges in $\Om_{h}$, where $\mathcal{E}^i_{h}$ denotes the set of all interior edges and $\mathcal{E}^b_{h}$ denotes the set of all boundary edges of $\Om_{h}$. We set $\Gamma_h:= \cup_{e \in \mathcal{E}^i_{h}} e$ to be union of all the inter-element boundaries of $\Om_{h}$. The set of the Lagrange nodal basis functions of $\V_{h}$ is $\{\phi_p,~p \in \mathcal{N}^i_{h}\}$. For $p \in \mathcal{N}_{h}$, let $\omega_p=\text{supp}(\phi_p)$ be the finite element star, i.e., the union of all the triangles in $\Om_{h}$ having $p$ as a common node. We denote by $\Gamma_p:= \Gamma_h \cap \omega_p$ to be the union of all interior edges in $\omega_p$.

Let $e = \partial K^1 \cap \partial K^2$ be a common edge of two triangles $K^1$ and $K^2$, $\mathbf{n}$ be the unit normal vector to $e$ pointing from $K^1$ to $K^2$, then we have
\begin{align}
\mathcal{J}_h:=[[ \partial_{h} w_h]]:= (\nabla w_h|_{K^1}- \nabla w_h|_{K^2}) \cdot \mathbf{n}. \label{jump}
\end{align}
Let $h_e$ be the length of an edge $e$ and $h_K$ be the diameter of $K$. Further, throughout the article, the notation $X \lesssim Y$ means that there exists generic positive constants $C$ such that $X \leq C Y.$ Next, we state the discrete trace inequality \cite[Section 10.3]{brenner2007mathematical} which will be frequently used in later analysis.
\begin{lemma}(Discrete trace inequality) \label{dti}
	Let $v \in H^1(K)$,  $K \in \Omega_h$ and $e$ be an edge of $K$. Then, it holds that
	\begin{eqnarray}\label{2.1}
	&\| v \|^2_{L^2{(e)}} \lesssim h_e^{-1}\| v \|^2_{L^2(K)}+h_e \| \nabla v\|^2_{L^2(K)}.
	\end{eqnarray}
\end{lemma}
We define $w_h : [0,T] \rightarrow \V_{h}$ to be the finite element approximation of $w$ which satisfy the following 
	\begin{align}
& (w_{h,t}(t),v_h-w_h(t))+\Bl{w_h(t),v_h-w_h(t)}\geq (f(t),v_h-w_h(t))~~ \forall \ v_h \in \mathcal{K}_{h,\chi}(t), \label{DCVI}\\
&w_h(0)=w_{h,0} \in \V_h. \notag 
\end{align}
For a.e. $t \in (0, T]$, we define the discrete version of $\sigma$ which will be used in the next subsection as
	\begin{align} \label{def:dissig}
(\sigma_h(t),v_h)=(w_{h,t}(t),v_h)+\Bl{w_h(t),v_h}-(f(t),v_h)~~ \forall \ v_h \in \V_h.
\end{align}
	See articles \cite{veeser2001efficient,fierro2003posteriori,nochetto2005fully} 
	for various choices of $\sigma_h$. 
\section{The Elliptic Reconstruction Map for Parabolic VIs} \label{sec4}
Let $t \in [0,T]$, and for any given finite element approximation $w_h \in \V_h$ and $z_h(t):= (f+ \sigma_{h} -w_{h,t})(t) \in L^2(\Omega)$, we introduce the {\it elliptic reconstruction} $\mathcal{W}(t):= \mathcal{R} w_h(t): \V_h \rightarrow \V$ of $w_h(t)$ which satisfies the following elliptic problem:
	\begin{align} \label{def:elrc}
(\nabla\mathcal{W}(t), \nabla v ) =  (z_h(t),v) \quad  \mbox{for all }v \in \V.
\end{align}
From the standard theory of elliptic problems \cite{ciarlet2002finite,brenner2007mathematical}, it is well known that the problem \eqref{def:elrc} has a unique solution. Moreover, the continuous elliptic reconstruction $\mathcal{W}$ and discrete solution $w_h$ satisfy the following orthogonality relation:
\begin{lemma} \label{Gal}
	It holds that
	\begin{align}
	\big(\nabla (\mathcal{W} -w_h)(t), \nabla v_h\big) =0 \quad \mbox{for all }v_h \in \V_h.  \label{galprop}
	\end{align}
	\end{lemma}
\begin{proof}
	Let $t \in [0,T]$ and $v_h \in \V_h$, then from equation $\eqref{def:dissig}$, it holds that
		\begin{align}
		&\big(\nabla \mathcal{W}(t), \nabla v_h\big) - \big( \nabla w_h(t), \nabla v_h \big) \notag \\ & = \big( f + \sigma_{h} -w_{h,t}, v_h\big) + \big(  w_{h,t} -f -\sigma_{h}, v_h \big) = 0. \notag
		\end{align}
	Hence, we have the proof of \eqref{galprop}.
	\end{proof}

\section{A Posteriori Error Analysis} \label{sec5}
	
This section provides the a posteriori error analysis in an abstract setting. As a starting point,
we introduce a lemma which provides an inequality for the difference between elliptic reconstruction $\mathcal{W}$ in \eqref{def:elrc} and $w$ solving \eqref{CVI}.		
	
\begin{lemma} \label{pointerror}
	Let $\mathcal{A} = -\Delta$ be the self adjoint operator, then for $0 < t \le T$ and $\forall~ v \in \V$, the following error equation holds in the the sense of distributions 
\begin{align}
(\mathcal{W}-w)_t + \mathcal{A} (\mathcal{W}-w) \leq \sigma_h + (\mathcal{W}-w_h)_t. \label{2}
\end{align}
Here $\mathcal{W}$, $w$ and $w_h$ respectively solve \eqref{def:elrc}, \eqref{CVI} and \eqref{DCVI}. 
\end{lemma}
\begin{proof}
For any $v \in \V$, we consider
\begin{align}\label{eq:diffWw}
	&	\langle (\mathcal{W}-w)_t, v \rangle_{-1,1} + \big( \nabla (\mathcal{W}-w), \nabla v\big) \notag
	\\ & = \langle \mathcal{W}_t, v \rangle_{-1,1} - \langle w_t, v \rangle_{-1,1} -\big( \nabla w, \nabla v\big) + \big( \nabla \mathcal{W}, \nabla v\big) \notag
	\\&  = \langle \mathcal{W}_t, v \rangle_{-1,1} -\Ian{\sigma(t),v}_{-1,1}-(f,v)+ \big( \nabla \mathcal{W}, \nabla v\big) ,
\end{align}
where in the last equality, we have used $w(t)$ from \eqref{def:sigma}. Subsequently, using
the definition of $\mathcal{W}$ from \eqref{def:elrc} in \eqref{eq:diffWw}, we obtain that
\begin{align*}	
	&	\langle (\mathcal{W}-w)_t, v \rangle_{-1,1} + \big( \nabla (\mathcal{W}-w), \nabla v\big)
	\\ &  = \langle \mathcal{W}_t, v \rangle_{-1,1} -\Ian{\sigma(t),v}_{-1,1}-(f,v) +(f+ \sigma_{h} -w_{h,t},v) \\
	& = -\langle \sigma, v \rangle_{-1,1} + \langle \sigma_h, v \rangle_{-1,1} + \langle (\mathcal{W}-w_h)_t, v  \rangle_{-1,1}.
\end{align*}
Hence, we have the following estimate for any $v \in \mathcal{V}$:
	\begin{align}
\langle (\mathcal{W}-w)_t,  v \rangle_{-1,1} + \big( \nabla (\mathcal{W}-w), \nabla v\big) + & \langle \sigma, v \rangle_{-1,1} \\  \notag & = \langle \sigma_h, v \rangle_{-1,1} + \langle (\mathcal{W}-w_h)_t, v  \rangle_{-1,1}. \notag
\end{align}
In particular, in the sense of distributions, we have 
\begin{align}
(\mathcal{W}-w)_t + \mathcal{A} (\mathcal{W}-w) + \sigma = \sigma_h + (\mathcal{W}-w_h)_t. \label{kprop}
\end{align}
Using the key sign property $\sigma \geq 0$ (cf.~Remark~\ref{rem:sign}), 
we conclude \eqref{2}.
\end{proof}
\subsection{Energy norm error estimates}
The next result follows immediately from Lemma \ref{pointerror}.
\begin{lemma} \label{lemma1}
	Let $\mathcal{W}$ and $w_h$ be the continuous elliptic reconstruction and discrete solution which satisfy \eqref{def:elrc} and \eqref{DCVI}, respectively. Then, the following bound holds for $t \in (0,T]$
	\begin{align}
	\frac{1}{2} \|&(\mathcal{W}-w)(t)\|^2_{L^2(\Om)}   + \frac{1}{2} \int_{0}^{t} \|(\mathcal{W}-w)(s)\|^2_{\V} ds \notag \leq \int_0^t \|\sigma_h\|_{\V^*} \|\mathcal{W}-w_h\|_{\V} \\ 
	& + \Big| \int_{0}^{t}   \langle \sigma_{h} , w_h - \chi \rangle_{-1,1} \Big|  +\int_0^t \|\sigma^-_{h}(s)\|^2_{\V^*} ds +  \int_{0}^{t} \|(w_h-\mathcal{W})_t(s) \|^2_{\V^*}  ds  \notag \\ 
	& \qquad \qquad + \Big|\int_{0}^{t}  \langle\sigma_{h}^- ,  \mathcal{W} - \chi \rangle_{-1,1} \Big|  + 	\frac{1}{2} \|\mathcal{W}(0)-w(0)\|^2_{L^2(\Om)}.  \label{key4}
	\end{align}
\end{lemma}
\begin{remark}
\rm 
	All the terms on the right hand side are expected to go to zero. In particular, 
	the second term will go to zero because the limiting multiplier fulfills the complementarity
	condition $\langle \sigma, \chi - w \rangle_{-1,1} = 0$, see Remark~\ref{rem:comp}. 
	Moreover, the third and fifth terms 
	are expected to vanish because the continuous multiplier $\sigma \ge 0$. This for instance
	immediately follows for the continuous piecewise linear finite elements 
	\cite{veeser2001efficient,nochetto2003pointwise}. \\
	Additionally, these terms also accounts for the variational inequality solver accuracy.
	Indeed, they are expected to be small once the inequality has been solved. 
\end{remark}
\begin{proof}
We multiply \eqref{2} by a test function $v:=\mathcal{W}-w\in \V $ to arrive at 
\begin{align}
 & \langle (\mathcal{W}-w)_t ,  v \rangle_{-1,1}  + (\nabla (\mathcal{W}-w), \nabla v) \leq \langle \sigma_{h} , v \rangle_{-1,1} + \langle (\mathcal{W}-w_h)_t, v \rangle_{-1,1} \notag \\ & \implies  \langle (\mathcal{W}-w)_t , \mathcal{W} -w\rangle_{-1,1} + a(\mathcal{W}-w, \mathcal{W}-w) \notag \\   &\hspace*{3cm} \leq \langle \sigma_{h} , \mathcal{W}-w \rangle_{-1,1} + \langle (\mathcal{W}-w_h)_t, \mathcal{W}-w \rangle_{-1,1} \notag \\ & \implies \frac{1}{2} \frac{d}{dt} \|\mathcal{W}-w\|_{L^2(\Omega)}^2 + a(\mathcal{W}-w,\mathcal{W}-w) \notag \leq  \\   &\hspace*{3cm} \langle \sigma_{h} , \mathcal{W}-w \rangle_{-1,1} + \langle (\mathcal{W}-w_h)_t, \mathcal{W}-w \rangle_{-1,1} \, . \label{key5} \\ & \implies \frac{1}{2} \frac{d}{dt} \|\mathcal{W}-w\|_{L^2(\Omega)}^2 + a(\mathcal{W}-w,\mathcal{W}-w) \notag \leq \langle \sigma_{h} , \mathcal{W}-w_h \rangle_{-1,1}  \\   &\hspace*{3cm} + \langle \sigma_{h} , w_h - \chi \rangle_{-1,1} + \langle \sigma_{h} , \chi - w \rangle_{-1,1}  + \langle (\mathcal{W}-w_h)_t, \mathcal{W}-w \rangle_{-1,1} \, . \label{key6}
\end{align}
Using the properties of $-\Delta$ operator, it holds that $$a(v, v)^{\frac{1}{2}}:= \|v\|_{\V} \quad \forall~ v \in \V$$ and for $t \in (0, T]$, we integrate \eqref{key6} on both sides from $0$ to $t$ 
\begin{align}
& \frac{1}{2}\big[\|\mathcal{W}(t)-w(t)\|^2_{L^2(\Om)} -\|\mathcal{W}(0)-w(0)\|^2_{L^2(\Om)} \big]  + \int_0^t \|\mathcal{W}-w\|^2_{\V}\notag \\ & \hspace*{2cm} \leq   \int_{0}^{t}  \langle \sigma_{h} , \mathcal{W}-w_h \rangle_{-1,1} + \int_{0}^{t} \langle \sigma_{h} , w_h - \chi \rangle_{-1,1} +\int_{0}^{t} \langle \sigma_{h} , \chi -w  \rangle_{-1,1} \notag \\  & \qquad \qquad \qquad \qquad + \int_0^t \langle (\mathcal{W}-w_h)_t, \mathcal{W} -w\rangle_{-1,1}  \\ & \implies \frac{1}{2} \|\mathcal{W}(t)-w(t)\|^2_{L^2(\Om)} + \int_0^t \|\mathcal{W}-w\|^2_{\V} \leq   \underbrace{\int_{0}^{t}  \langle \sigma_{h} , \mathcal{W}-w_h \rangle_{-1,1}}_\text{Term 1}  \notag \\ &  
\qquad + \int_{0}^{t}  \langle \sigma_{h} , w_h - \chi \rangle_{-1,1} + \underbrace{\int_{0}^{t}  \langle \sigma_{h} ,  \chi - w \rangle_{-1,1}}_\text{Term 2} +  \underbrace{\int_0^t \langle (\mathcal{W}-w_h)_t, \mathcal{W} -w\rangle_{-1,1} }_\text{Term 3} \\ &  \qquad \qquad \qquad + \frac{1}{2} \|\mathcal{W}(0)-w(0)\|^2_{L^2(\Om)} \, .  \label{key3}
\end{align}
Using $v^+:= \max\{v,0\}$, $v^-:= \max\{-v,0\}$, Cauchy Schwarz inequality and the following generalized Young's inequality
\begin{align}
ab \leq \frac{a^2}{2 \epsilon} + \frac{\epsilon}{2} b^2 \quad \forall a, b \in \mathbb{R} \label{young},
\end{align}
we derive the following bounds 
\begin{align}
\mbox{Term 1} & \leq \int_0^t \|\sigma_h\|_{\V^*} \|\mathcal{W}-w_h\|_{\V}. \label{key1}
\end{align}
\begin{align}
\mbox{Term 2} & = \int_{0}^{t}  \langle \sigma_{h} ,  \chi - w \rangle_{-1,1} \notag \\ & = \int_{0}^{t}  \langle \sigma_{h}^+ - \sigma_{h}^- ,  \chi - w \rangle_{-1,1}  \leq \int_{0}^{t}  \langle\sigma_{h}^- ,  w - \chi \rangle_{-1,1} \quad (\mbox{using } \sigma_{h}^+ \geq 0 \mbox{ and }w \geq \chi) \notag \\ & = \int_{0}^{t}  \langle \sigma_{h}^- ,  w -\mathcal{W} \rangle_{-1,1} + \int_{0}^{t}  \langle\sigma_{h}^- ,  \mathcal{W} - \chi \rangle_{-1,1} \notag \\ & \leq \int_0^t \|\sigma_{h}^-\|_{\V^*} \|\mathcal{W}-w\|_{\V} +\int_{0}^{t}  \langle\sigma_{h}^- ,  \mathcal{W} - \chi \rangle_{-1,1} \notag \\ &  \leq \int_0^t \|\sigma_{h}^-\|^2_{\V^*} +  \frac{1}{4} \int_0^t \|\mathcal{W}-w\|^2_{\V} +\int_{0}^{t}  \langle\sigma_{h}^- ,  \mathcal{W} - \chi \rangle_{-1,1}  \quad (\epsilon =2 \mbox{ in } \eqref{young}). \label{key8}
\end{align}
\begin{align}
\mbox{Term 3} & \leq \int_0^t\|(\mathcal{W}-w_h)_t\|_{\V^*} \|\mathcal{W}-w\|_{\V} \notag \\ & \leq \int_0^t \|(\mathcal{W}-w_h)_t\|^2_{\V^*} +  \frac{1}{4} \int_0^t \|\mathcal{W}-w\|^2_{\V} \quad (\epsilon =2 \mbox{ in } \eqref{young}). \label{key2}
\end{align}
We use \eqref{key1} \eqref{key8} and \eqref{key2} in \eqref{key3} to have the desired estimate \eqref{key4}.
	\end{proof}
Next, we introduce $\mathcal{P}_h : \V \rightarrow \V_h$ to be the elliptic projection operator, i.e.,
\begin{align} \label{def:ellproj}
a(\mathcal{P}_hw, v_h)=a(w,v_h) \quad \mbox{for all }v_h \in \V_h.
\end{align}
For a given $g \in L^2(\Omega)$, let us suppose $u \in \V $ and $u_h \in \V_h$ satisfy the following elliptic equations, respectively
\begin{align}
a(u,v)  &=  (g, v) \quad \mbox{for all }v \in \V, \label{cts}\\ 	a(u_h,v_h) &= (g, v_h)  \quad \mbox{for all }v_h \in \V_h. \label{disc}
\end{align}
From \eqref{cts} and \eqref{disc}, we note that $a(u-u_h,v_h) =0 \quad \forall v_h \in \V_h \implies u_h=\mathcal{P}_h u$. In addition, for a given positive constant $C$, the following known a posteriori error estimate holds \cite{verfurth1996review,MR2249676}
\begin{align} 
\|u-u_h\|_{\mathcal{X}} &\leq C \eta(u_h,g:\mathcal{X}) \label{eq11} 
\end{align}
where 
$C$ 
is a constant which depend only on the properties of an elliptic a posteriori error estimator and $\eta:=\eta(u_h,g:\mathcal{X})$ is the computable a posteriori error estimator which depends only on the space $\mathcal{X}:= L^2(\Omega), \V$ or $\V^*$ and the given data $u_h$ and $g$. Next, we note that $\mathcal{P}_h\mathcal{W}=w_h$ using Lemma \ref{Gal}. Therefore, using similar ideas, we have the following a posteriori error estimate 
\begin{align}
\|(w_h - \mathcal{W})(t)\|_{L^2(\Omega)} & \leq C \eta(w_h,z_h:L^2(\Omega))  \label{eq10} \\\|(w_h - \mathcal{W})(t)\|_{\V} & \leq C \eta(w_h,z_h:\V) \, .  \label{eq13} 
\end{align}
We differentiate \eqref{def:elrc} with respect to $t$ to get the following
\begin{align} \label{def1}
a(\mathcal{W}_t,v) =  (z_{h,t},v) \quad  \mbox{for all }v \in \V.
\end{align}
Using \eqref{galprop}, \eqref{def:ellproj}, and \eqref{eq10}, there holds 
 \begin{align} 
\mathcal{P}_h\mathcal{W}_t=w_{h,t} \, , & \quad \|(w_h - \mathcal{W})_t\|_{\V^*} \leq C \eta(w_{h,t},z_{h,t}:\V^*) \, . 
\label{eq2} 
\end{align}

We are now ready to state the main a posteriori result of this section.
\begin{theorem}(\textbf{Abstract a posteriori error estimate in energy norm}) \label{main1}
	Let $w$ and $w_h$ be the continuous and discrete solutions which satisfy \eqref{CVI} and \eqref{DCVI}, respectively. Then, for a positive constant $C$, the following a posteriori error estimate hold for $0 < t \le T$:
		\begin{align}
		\frac{1}{2} \|(w-w_h)(t)\|_{L^2(\Om)} 
		& \leq  \bigg(\int_{0}^{t} \|\sigma^-_{h}(s) \|^2_{\V^*}  ds \bigg)^{\frac{1}{2}} 
		+ \bigg( \Big| \int_{0}^{t}  \langle \sigma_{h} , w_h - \chi \rangle_{-1,1} \Big| \bigg)^{\frac{1}{2}}  \notag \\ 
		& \quad
		+ \bigg( \Big| \int_{0}^{t} \langle\sigma_{h}^- ,  \mathcal{W} - \chi \rangle_{-1,1} \Big|   \bigg)^{\frac{1}{2}}	
		+ \bigg(\int_0^t \|\sigma_h\|_{\V^*} \eta(w_{h},z_{h}:\V) \bigg)^{\frac{1}{2}} 	 \notag \\ 
		& \quad
		+ \frac{1}{2}\eta(w_h,z_h:L^2(\Omega))		
		+ \bigg(\int_{0}^{t} \eta(w_{h,t},z_{h,t}:\V^*)^2  ds \bigg)^{\frac{1}{2}}	\notag \\
		& \quad	
		+ \frac{1}{2} \|w_h(0)-w_0\|_{L^2(\Om)}  + \frac{1}{2} \eta(w_h(0), z_h(0):L^2(\Omega)). \label{eq}
	\end{align}
\end{theorem}
\begin{proof}
	Let us consider
\begin{align}
\frac{1}{2}\|(w_h-w)(t)\|_{L^2(\Om)} &\leq \frac{1}{2} \|(w_h- \mathcal{W})(t)\|_{L^2(\Om)} + \frac{1}{2}\| (\mathcal{W}- w)(t)\|_{L^2(\Om)} \notag \\ 
& \hspace*{-3cm}\leq \frac{1}{2} \|(w_h- \mathcal{W})(t)\|_{L^2(\Om)} + \bigg(\int_0^t \|\sigma_h\|_{\V^*} \|\mathcal{W}-w_h\|_{\V}\bigg)^{\frac{1}{2}}  + \bigg(\int_{0}^{t} \|\sigma^-_{h}(s) \|^2_{\V^*}  ds \bigg)^{\frac{1}{2}}  \notag \\ 
&\hspace*{-3cm}  + \bigg( \Big| \int_{0}^{t} \langle \sigma_{h} , w_h - \chi \rangle_{-1,1}\Big| \bigg)^{\frac{1}{2}}  +  \bigg( \int_{0}^{t} \|(w_h-\mathcal{W})_t(s) \|^2_{\V^*}  ds \bigg)^{\frac{1}{2}} + 	\frac{1}{2} \|\mathcal{W}(0)-w(0)\|_{L^2(\Om)} \notag \\ 
& + \bigg( \Big| \int_{0}^{t}  \langle\sigma_{h}^- ,  \mathcal{W} - \chi \rangle_{-1,1} \Big| \bigg)^{\frac{1}{2}}  \label{eq1}
\end{align}
where we used Lemma \ref{lemma1} in the second step. In addition, using bounds \eqref{eq10}, \eqref{eq13} and \eqref{eq2}, we have 
\begin{align}
&\frac{1}{2}\|(w_h-w)(t)\|_{L^2(\Om)}  \leq \frac{1}{2} \|(w_h- \mathcal{W})(t)\|_{L^2(\Om)} + \bigg(\int_{0}^{t} \|\sigma^-_{h}(s) \|^2_{\V^*}  ds \bigg)^{\frac{1}{2}} \notag  \\ 
& \hspace*{0.5cm}+  \bigg(\int_0^t \|\sigma_h\|_{\V^*} \eta(w_{h},z_{h}:\V) \bigg)^{\frac{1}{2}} 
+ \bigg( \Big| \int_{0}^{t} \langle \sigma_{h} , w_h - \chi \rangle_{-1,1} \Big| \bigg)^{\frac{1}{2}} 
+ \bigg( \int_{0}^{t} \|(w_h-\mathcal{W})_t(s) \|^2_{\V^*}  ds \bigg)^{\frac{1}{2}} \notag \\ 
& \hspace*{0.5cm}+ \frac{1}{2} \|\mathcal{W}(0)-w_h(0)\|_{L^2(\Om)} + \frac{1}{2} \|w_h(0)-w(0)\|_{L^2(\Om)} 
+ \bigg( \Big| \int_{0}^{t} \langle\sigma_{h}^- ,  \mathcal{W} - \chi \rangle_{-1,1} \Big|  \bigg)^{\frac{1}{2}}\notag \\ 
& \leq \frac{1}{2}\eta(w_h,z_h:L^2(\Omega)) 
+ \bigg(\int_{0}^{t} \|\sigma^-_{h}(s) \|^2_{\V^*}  ds \bigg)^{\frac{1}{2}} 
+ \bigg(\int_0^t \|\sigma_h\|_{\V^*} \eta(w_{h},z_{h}:\V) \bigg)^{\frac{1}{2}}  \notag \\ 
& \hspace*{0.5cm}+ \bigg(\Big| \int_{0}^{t} \langle \sigma_{h} , w_h - \chi \rangle_{-1,1} \Big| \bigg)^{\frac{1}{2}} 
+ \bigg(\int_{0}^{t} \eta(w_{h,t},z_{h,t}:\V^*)^2  ds \bigg)^{\frac{1}{2}} 
+ \bigg( \Big| \int_{0}^{t} \langle\sigma_{h}^- ,  \mathcal{W} - \chi \rangle_{-1,1} \Big|  \bigg)^{\frac{1}{2}} \notag \\ 
& \hspace*{0.5cm}+ \frac{1}{2} \|w_h(0)-w_0\|_{L^2(\Om)} +  \frac{1}{2} \eta(w_h(0),z_h(0):L^2(\Omega)). \notag 
\end{align}
	\end{proof}

\begin{remark}
\rm 
	It is easy to see that for $v \in \mathcal{V}$, we have $\langle \sigma_h, v \rangle_{-1,1} 
	= (\sigma_h, v)$. Once this is used in the proof of Lemma \ref{lemma1}, the estimate in \eqref{key4} still remains true with $\|\sigma_h\|_{\mathcal{V}^*}$ replaced by $\|\sigma_h\|_{L^2(\Omega)}$. Different choices for $\sigma_{h}$ (eg., the ``broken" nodal multipliers \cite{fierro2003posteriori,nochetto2003pointwise} or using the partition of unity \cite{nochetto2005fully}) could give rise to different bounds for $\|\sigma_h\|_{\mathcal{V}^*}$ in estimate \eqref{eq}. 
\end{remark}
\section{Applications} \label{sec6}
In this section, we will provide an explicit form of $\eta(w_{h,t},z_{h,t}:\V^*)$ mentioned in estimate \eqref{eq}. In other words, we will bound $\|(w_h - \mathcal{W})_t\|_{\V^*}$ using \eqref{eq2}. Let $t \in (0, T)$ and $p:=(w_h - \mathcal{W})_t$, we first recall the definition of dual norm
\begin{align}
\|p\|_{\V^*}: =\sup_{\|v\|_\V \leq 1} \langle p, v \rangle_{-1,1}.
\end{align}
Let $\Omega$ be sufficiently smooth and using the duality arguments, we state the next remark which will be useful to derive the residual-type error estimate.
\begin{remark} \label{1}
		For any given $\chi \in \V$ and using the standard duality arguments, let us define $\phi \in \V$ by
	\begin{align}
a(\phi,v) =	\langle v, \chi \rangle_{-1,1} \quad \forall v \in \V \, . \label{eq3}
	\end{align}
	Also there exists a constant $C_1> 0$ which depends on $\Omega$ such that 
	\begin{align}
	\|\phi\|_{H^3(\Omega)} \leq C_1 \|\chi\|_{H^1(\Omega)}. \label{eq4}
	\end{align}
	\end{remark}
 Using definition \eqref{jump}, we denote $\mathcal{J}_{h,t}:=[[ \partial_{h} w_{h,t}]]$. In the next lemma, we provide the bound on the term $\|(w_h - \mathcal{W})_t\|_{\V^*}$.
\begin{lemma} \label{main2}
	Let $\Omega$ be sufficiently smooth and $k \geq 2$ be the polynomial degree, then it holds that
	\begin{align}
	\|p\|_{\V^*}:=\|(w_h - \mathcal{W})_t\|_{\V^*} \leq C \eta^1_h(w_{h,t}),
	\end{align}
	where $\eta^1_h(w_{h,t})^2:=  \sum_{K \in \Omega_h} h_K^6 \| (f+ \sigma_{h} + \Delta w_{h,t} -w_{h} )_t\|_{L^2(K)}^2 + \sum_{e \in \Gamma_h} h_K^5 \|\mathcal{J}_{h,t}\|_{L^2(e)}^2.$
	\end{lemma}
\begin{proof}
	Using the definitions of $p$, $\phi$ (Remark \ref{1}) and suitable construction of interpolation operator $\Pi_{h}: \mathcal{V} \rightarrow \mathcal{V}_h$ (for instance Clement's interpolant \cite{brenner2007mathematical}), we consider the following 
	\begin{align*}
\langle p, \chi \rangle_{-1,1} & = a(\phi,p) \\ & =a(\phi- \Pi_h \phi, p) \quad (\text{using \eqref{eq2}} \mbox{ i.e., Galerkin Orthogonality}) \\ & \leq  \sum_{K \in \Omega_h} \int_{K} \big|(\phi- \Pi_h \phi) \Delta p \big| dx + \sum_{e \in \Gamma_h} \int_e \big|(\phi- \Pi_h \phi) [[ \partial_{h} p]]\big| \\ & =  \sum_{K \in \Omega_h} \int_{K} \big|(\phi- \Pi_h \phi) \Delta p \big| dx + \sum_{e \in \Gamma_h} \int_e \big|(\phi- \Pi_h \phi) \big| \big| [[ \partial_{h} w_{h,t}]] \big|
	\end{align*}
	where we used that $[[ \partial_{h} \mathcal{W}_{t}]]=0$ (in the last step) using the elliptic regularity. Then, we have
	\begin{align}
	\langle p, \chi \rangle_{-1,1} & =  \sum_{K \in \Omega_h} \int_{K} h_K^{-3}\big|(\phi- \Pi_h \phi) \big| h_K^3 \big|\Delta p \big| dx \notag   \\ & \hspace*{2cm} + \sum_{e \in \Gamma_h} \int_e h_e^{-\frac{5}{2}} \big|(\phi- \Pi_h \phi) \big| h_e^{\frac{5}{2}}  \big| [[ \partial_{h} w_{h,t}]] \big| ds \notag \\ & \leq \sum_{K \in \Omega_h} h_K^{-3} \|\phi- \Pi_h \phi\|_{L^2(K)} h_K^3\|\Delta p\|_{L^2(K)} \notag  \\ & \hspace*{2cm} + \sum_{e \in \Gamma_h}  h_e^{-\frac{5}{2}} \|\phi- \Pi_h \phi\|_{L^2(e)} h_e^{\frac{5}{2}} \|[[ \partial_{h} w_{h,t}]]  \|_{L^2(e)}. \label{3}
	\end{align}
	Using the following relation
	\begin{align*}
	\Delta p=\Delta(w_h - \mathcal{W})_t & = \Delta w_{h,t} - \Delta \mathcal{W}_{t} \\ & = \Delta w_{h,t} + (f+ \sigma_{h} -w_{h,t})_t \quad (\text{using \eqref{def:elrc}})  \\ & = (f+ \sigma_{h} + \Delta w_{h,t} -w_{h} )_t
	\end{align*}
	and the standard interpolation estimates \cite[Section 4.4]{brenner2007mathematical}, discrete trace inequality (Lemma \ref{dti}) and estimate \eqref{eq4}, we conclude from \eqref{3} 
		\begin{align*}
	\langle p, \chi \rangle_{-1,1} & \leq C_2 \eta^1_h(w_{h,t}) \| \phi \|_{H^3(\Omega)} \\ \implies \|p\|_{\V^*} & \leq C \eta^1_h(w_{h,t}),
	\end{align*}
	where $C_2>0$ is an interpolation estimate constant and $C= C_1 C_2$.
	\end{proof}
Using similar arguments, we have the following lemma for $k \geq 2$.
\begin{lemma} \label{main3}
	Let $C>0$ be a positive constant, then the following holds
	\begin{align}
	\eta(w_h,z_h:L^2(\Omega)) \leq C \eta^0_h(w_{h}), \label{ref}
	\end{align}
	where $\eta^0_h(w_{h})^2:=  \sum_{K \in \Omega_h} h_K^4 \| f+ \sigma_{h} + \Delta w_{h,t} -w_{h} \|_{L^2(K)}^2 + \sum_{e \in \Gamma_h} h_K^3 \|\mathcal{J}_{h}\|_{L^2(e)}^2.$
	\end{lemma}
Finally, we use Lemmas \ref{main2} and \ref{main3} in Theorem \ref{main1} to state the residual type a posteriori error estimate for the displacements.
\begin{theorem}
	Let $t \in (0,T]$ and $\Omega$ be sufficiently smooth. If $k \geq 2$, then, the following a posteriori error estimate hold for $0 < t \le T$:
	\begin{align}
	\frac{1}{2} \|w(t)-w_h(t)\|_{L^2(\Om)} & \lesssim  \frac{1}{2} \eta^0_h(w_{h}) + \bigg(\int_{0}^{t} \|\sigma^-_{h}(s) \|^2_{\V^*}  ds \bigg)^{\frac{1}{2}} + \bigg(\int_{0}^{t} (\eta^1_h(w_{h,t}(s)))^2  ds \bigg)^{\frac{1}{2}}  \notag \\ & \hspace*{-2.5cm} + \frac{1}{2} \|w_h(0)-w_0\|_{L^2(\Om)} +\frac{1}{2} \eta^0_h(w_{h}(0)) + \bigg(\int_0^t \|\sigma_h\|_{\V^*} \eta^0_h(w_{h})  \bigg)^{\frac{1}{2}}  \notag \\ 
	&  \hspace*{-2.5cm}+ \bigg( \Big| \int_{0}^{t}  \langle \sigma_{h} , w_h - \chi \rangle_{-1,1} \Big| \bigg)^{\frac{1}{2}} 
	+ \bigg( \Big| \int_{0}^{t} \langle\sigma_{h}^- ,  \mathcal{W} - \chi \rangle_{-1,1} \Big| \bigg)^{\frac{1}{2}}.
	\end{align}
\end{theorem}
	Let $k =1$ be the polynomial degree as discussed in \cite{makridakis2003elliptic}, due to the lack of superconvergence in $\V^*$, the use of negative norm does not give better results. Hence, we use the following estimate using \eqref{ref}
	\begin{align}
	L^2(\Omega) \hookrightarrow \V^* \text{ and } \eta(w_{h,t},z_{h,t}:\V^*) \leq \eta(w_{h,t},z_{h,t}:L^2(\Omega)) \leq C \eta^0_h(w_{h,t}) ,
	\end{align}
	where the last inequality holds due to the classical duality argument, see \cite[Sec.~3.3]{MR1442375}. 
	In view of Theorem \ref{main1}, we obtain the following a posteriori error estimate for $k =1$ 
		\begin{align}
	\frac{1}{2} \|w(t)-w_h(t)\|_{L^2(\Om)} & \lesssim  \frac{1}{2} \eta^0_h(w_{h}) + \bigg(\int_{0}^{t} \|\sigma^-_{h}(s) \|^2_{\V^*}  ds \bigg)^{\frac{1}{2}} + \bigg(\int_{0}^{t} (\eta^0_h(w_{h,t}(s)))^2  ds \bigg)^{\frac{1}{2}}  \notag \\ 
	& \hspace*{-2.5cm} + \frac{1}{2} \|w_h(0)-w_0\|_{L^2(\Om)} +\frac{1}{2} \eta^0_h(w_{h}(0)) + \bigg(\int_0^t \|\sigma_h\|_{\V^*} \eta^0_h(w_{h})  \bigg)^{\frac{1}{2}}  \notag \\ 
	&  \hspace*{-2.5cm}+ \bigg(\Big| \int_{0}^{t} \langle \sigma_{h} , w_h - \chi \rangle_{-1,1} \Big| \bigg)^{\frac{1}{2}} + \bigg(\Big| \int_{0}^{t} \langle\sigma_{h}^- ,  \mathcal{W} - \chi \rangle_{-1,1} \Big| \bigg)^{\frac{1}{2}}.
	\end{align}

\section*{Acknowledgement}	
We are thankful to Mahamadi Warma for suggestions on Remark~\ref{rem:sign} and several discussions.
	
 \bibliographystyle{plain}
\bibliography{AK}

\begin{thebibliography}{10}

\bibitem{achdou2008posteriori}
Yves Achdou, Fr{\'e}d{\'e}ric Hecht, and David Pommier.
\newblock A posteriori error estimates for parabolic variational inequalities.
\newblock {\em Journal of Scientific Computing}, 37:336--366, 2008.

\bibitem{MR1442375}
Mark Ainsworth and J.~Tinsley Oden.
\newblock A posteriori error estimation in finite element analysis.
\newblock {\em Comput. Methods Appl. Mech. Engrg.}, 142(1-2):1--88, 1997.

\bibitem{ainsworth2011posteriori}
Mark Ainsworth and J~Tinsley Oden.
\newblock {\em A posteriori error estimation in finite element analysis},
  volume~37.
\newblock John Wiley \& Sons, 2011.

\bibitem{ainsworth1993local}
Mark Ainsworth, J~Tinsley Oden, and CY~Lee.
\newblock Local a posteriori error estimators for variational inequalities.
\newblock {\em Numerical methods for partial differential equations},
  9(1):23--33, 1993.

\bibitem{bartels2004averaging}
S{\"o}ren Bartels and Carsten Carstensen.
\newblock Averaging techniques yield reliable a posteriori finite element error
  control for obstacle problems.
\newblock {\em Numerische Mathematik}, 99:225--249, 2004.

\bibitem{MR4593215}
S\"{o}ren Bartels and Alex Kaltenbach.
\newblock Explicit and efficient error estimation for convex minimization
  problems.
\newblock {\em Math. Comp.}, 92(343):2247--2279, 2023.

\bibitem{MR2772095}
S\"{o}ren Bartels and R\"{u}diger M\"{u}ller.
\newblock Quasi-optimal and robust a posteriori error estimates in
  {$L^\infty(L^2)$} for the approximation of {A}llen-{C}ahn equations past
  singularities.
\newblock {\em Math. Comp.}, 80(274):761--780, 2011.

\bibitem{brenner2007mathematical}
Susanne Brenner and Ridgway Scott.
\newblock {\em The mathematical theory of finite element methods}, volume~15.
\newblock Springer Science \& Business Media, 2007.

\bibitem{MR0428137}
Ha\"{\i}m Br\'{e}zis.
\newblock Probl\`emes unilat\'{e}raux.
\newblock {\em J. Math. Pures Appl. (9)}, 51:1--168, 1972.

\bibitem{chen2000residual}
Zhiming Chen and Ricardo~H Nochetto.
\newblock Residual type a posteriori error estimates for elliptic obstacle
  problems.
\newblock {\em Numerische Mathematik}, 84:527--548, 2000.

\bibitem{ciarlet2002finite}
Philippe~G Ciarlet.
\newblock {\em {The finite element method for elliptic problems}}.
\newblock SIAM, 2002.

\bibitem{eriksson1991adaptive}
Kenneth Eriksson and Claes Johnson.
\newblock {Adaptive finite element methods for parabolic problems I: A linear
  model problem}.
\newblock {\em SIAM Journal on Numerical Analysis}, 28(1):43--77, 1991.

\bibitem{eriksson1995adaptive}
Kenneth Eriksson and Claes Johnson.
\newblock Adaptive finite element methods for parabolic problems. {II}.
  {O}ptimal error estimates in {$L_\infty L_2$} and {$L_\infty L_\infty$}.
\newblock {\em SIAM J. Numer. Anal.}, 32(3):706--740, 1995.

\bibitem{fierro2003posteriori}
Francesca Fierro and Andreas Veeser.
\newblock A posteriori error estimators for regularized total variation of
  characteristic functions.
\newblock {\em SIAM journal on numerical analysis}, 41(6):2032--2055, 2003.

\bibitem{glowinski1980numerical}
Roland Glowinski.
\newblock {\em Numerical methods for nonlinear variational problems}.
\newblock Tata Institute of Fundamental Research, 1980.

\bibitem{MR0737005}
Roland Glowinski.
\newblock {\em Numerical methods for nonlinear variational problems}.
\newblock Springer Series in Computational Physics. Springer-Verlag, New York,
  1984.

\bibitem{MR4029105}
Thirupathi Gudi and Papri Majumder.
\newblock Conforming and discontinuous {G}alerkin {FEM} in space for solving
  parabolic obstacle problem.
\newblock {\em Comput. Math. Appl.}, 78(12):3896--3915, 2019.

\bibitem{MHintermueller_KIto_KKunisch_2002a}
M.~Hinterm{\"u}ller, K.~Ito, and K.~Kunisch.
\newblock The primal-dual active set strategy as a semismooth {N}ewton method.
\newblock {\em SIAM J. Optim.}, 13(3):865--888 (2003), 2002.

\bibitem{hoppe1994adaptive}
Ronald~HW Hoppe and Ralf Kornhuber.
\newblock Adaptive multilevel methods for obstacle problems.
\newblock {\em SIAM journal on numerical analysis}, 31(2):301--323, 1994.

\bibitem{lakkis2006elliptic}
Omar Lakkis and Charalambos Makridakis.
\newblock Elliptic reconstruction and a posteriori error estimates for fully
  discrete linear parabolic problems.
\newblock {\em Mathematics of computation}, 75(256):1627--1658, 2006.

\bibitem{makridakis2003elliptic}
Charalambos Makridakis and Ricardo~H Nochetto.
\newblock Elliptic reconstruction and a posteriori error estimates for
  parabolic problems.
\newblock {\em SIAM journal on numerical analysis}, 41(4):1585--1594, 2003.

\bibitem{moon2007posteriori}
Kyoung-Sook Moon, Ricardo~H Nochetto, Tobias Von~Petersdorff, and Chen-song
  Zhang.
\newblock A posteriori error analysis for parabolic variational inequalities.
\newblock {\em ESAIM: Mathematical Modelling and Numerical Analysis},
  41(3):485--511, 2007.

\bibitem{MR2249676}
Ricardo~H. Nochetto, Alfred Schmidt, Kunibert~G. Siebert, and Andreas Veeser.
\newblock Pointwise a posteriori error estimates for monotone semi-linear
  equations.
\newblock {\em Numer. Math.}, 104(4):515--538, 2006.

\bibitem{nochetto2003pointwise}
Ricardo~H Nochetto, Kunibert~G Siebert, and Andreas Veeser.
\newblock Pointwise a posteriori error control for elliptic obstacle problems.
\newblock {\em Numerische Mathematik}, 95(1):163--195, 2003.

\bibitem{nochetto2005fully}
Ricardo~H Nochetto, Kunibert~G Siebert, and Andreas Veeser.
\newblock Fully localized a posteriori error estimators and barrier sets for
  contact problems.
\newblock {\em SIAM journal on numerical analysis}, 42(5):2118--2135, 2005.

\bibitem{nochetto2010posteriori}
Ricardo~H Nochetto, Tobias von Petersdorff, and Chen-Song Zhang.
\newblock A posteriori error analysis for a class of integral equations and
  variational inequalities.
\newblock {\em Numerische Mathematik}, 116(3):519--552, 2010.

\bibitem{veeser2001efficient}
Andreas Veeser.
\newblock Efficient and reliable a posteriori error estimators for elliptic
  obstacle problems.
\newblock {\em SIAM journal on numerical analysis}, 39(1):146--167, 2001.

\bibitem{verfurth1996review}
R{\"u}diger Verf{\"u}rth.
\newblock A review of a posteriori error estimation.
\newblock In {\em and Adaptive Mesh-Refinement Techniques, Wiley \& Teubner}.
  Citeseer, 1996.

\bibitem{wheeler1973priori}
Mary~Fanett Wheeler.
\newblock {A priori $L_2$ error estimates for Galerkin approximations to
  parabolic partial differential equations}.
\newblock {\em SIAM Journal on Numerical Analysis}, 10(4):723--759, 1973.

\bibitem{wilmott1993option}
Paul Wilmott, Jeff Dewynne, and Sam Howison.
\newblock {Option pricing: Mathematical models and computation}.
\newblock {\em (No Title)}, 1993.

\end{thebibliography}
\end{document}